\documentclass[11pt,reqno]{amsart}
 \usepackage{amssymb,amsfonts,amscd,amsbsy, color, amsmath}
\usepackage[mathscr]{eucal}
\usepackage{url}
\usepackage{graphicx}
\usepackage{mathtools}

\newtheorem{theorem}{Theorem}[section]
\newtheorem{lemma}[theorem]{Lemma}

\newtheorem{corollary}[theorem]{Corollary}
\theoremstyle{definition}
\newtheorem{remark}[theorem]{Remark}
\numberwithin{equation}{section}

\newcommand{\sgn}{\operatorname{sgn}}

\makeatletter
\def\imod#1{\allowbreak\mkern5mu({\operator@font mod}\,\,#1)}
\makeatother
\allowdisplaybreaks

\makeatletter
\@namedef{subjclassname@2020}{%
  \textup{2020} Mathematics Subject Classification}
\makeatother

\begin{document}

\title[Quantum modularity of partial theta series with periodic coefficients]{Quantum modularity of partial theta series with periodic coefficients}

\author{Ankush Goswami}

\author{Robert Osburn}

\address{Research Institute for Symbolic Computation (RISC), Johannes Kepler University, 4040 Linz, Austria}

\email{ankushgoswami3@gmail.com\\ankush.goswami@risc.jku.at}

\address{School of Mathematics and Statistics, University College Dublin, Belfield, Dublin 4, Ireland}

\email{robert.osburn@ucd.ie}



\date{\today}

\begin{abstract}
We explicitly prove the quantum modularity of partial theta series with even or odd periodic coefficients. As an application, we show that the Kontsevich-Zagier series $\mathscr{F}_t(q)$ which matches (at a root of unity) the colored Jones polynomial for the family of torus knots $T(3,2^t)$, $t \geq 2$, is a weight $3/2$ quantum modular form.
This generalizes Zagier's result on the quantum modularity for the ``strange" series $F(q)$.
\end{abstract}

\maketitle

\section{Introduction}
In \cite{z}, Zagier introduced the notion of a quantum modular form of weight $k \in \frac{1}{2} \mathbb{Z}$ as a function $g : \mathbb{Q} \rightarrow \mathbb{C}$ for which the function $r_{\gamma} : \mathbb{Q} \setminus \{\gamma^{-1}(i \infty) \} \rightarrow \mathbb{C}$ given by
\begin{equation*} \label{og}
g(\alpha) - (c\alpha+d)^{-k} g\Bigl( \frac{a\alpha+b}{c\alpha+d} \Bigr) =: r_{\gamma}(\alpha)
\end{equation*}
extends to a real-analytic function on $\mathbb{P}^{1}(\mathbb{R}) \setminus S_{\gamma}$, where $S_{\gamma}$ is a finite set, for each $\gamma = \begin{pmatrix}
a & b\\c & d\end{pmatrix} \in SL_2(\mathbb{Z})$. Suitable modifications can be made to restrict the domain of $r_{\gamma}$ to appropriate subsets of $\mathbb{Q}$ and allow both multiplier systems and transformations on subgroups of $SL_2(\mathbb{Z})$. Since their inception, there has been substantial interest in studying these modular objects which emerge in diverse contexts: Maass forms \cite{b}, supersymmetric quantum field theory \cite{djr}, topological invariants for plumbed 3-manifolds \cite{bmm}, \cite{ccfgh}, \cite{cfs}, combinatorics \cite{fjks}, \cite{kll}, unified Witten-Reshetikhin-Turaev invariants \cite{hl2} and $L$-functions \cite{lz}, \cite{n}. For more examples, see Chapter 21 in \cite{bfor}. 

One of the most influential of the original five examples from \cite{z} is the Kontsevich-Zagier ``strange" series \cite{z1}
\begin{equation} \label{kz}
F(q) := \sum_{n \geq 0} (q)_n
\end{equation}
where
\begin{equation*}
(a_1, a_2, \dotsc, a_j)_n = (a_1, a_2, \dotsc, a_j ; q)_n := \prod_{k=1}^{n}(1-a_1 q^{k-1})(1 - a_2 q^{k-1}) \cdots (1-a_j q^{k-1})
\end{equation*} 
is the standard $q$-hypergeometric notation, valid for $n \in \mathbb{N}_{0} \cup \{ \infty \}$. $F(q)$ is ``strange" in the sense that it does not converge on any open subset of $\mathbb{C}$, but is well-defined when $q$ is a root of unity (where it is finite). Zagier proves that for $\alpha \in \mathbb{Q}$, $\phi(\alpha):= e^{\frac{\pi i \alpha}{12}} F(e^{2\pi i \alpha})$ is a quantum modular form of weight $3/2$ on $\mathbb{Q}$ with respect to $SL_2(\mathbb{Z})$. The key to proving this result is  the ``strange identity"
\begin{equation} \label{strid}
F(q) ``=" -\frac{1}{2} \sum_{n \geq 1} n \Bigl( \frac{12}{n} \Bigr) q^{\frac{n^2 - 1}{24}}
\end{equation}
where $``="$ means that the two sides agree to all orders at every root of unity (for further details, see Sections 2 and 5 in \cite{z1}) and $\bigl( \frac{12}{*} \bigr)$ is the quadratic character of conductor $12$. The idea is to prove quantum modular properties for the right-hand side of (\ref{strid}) which are then inherited by $F(q)$. The purpose of this paper is to place the right-hand side of (\ref{strid}) and other examples in the literature into the general context of quantum modularity of partial theta series with even or odd periodic coefficients. Before stating our main result, we introduce some notation. 

Let $f: \mathbb{Z} \rightarrow \mathbb{C}$ be an even or odd function with period $M\geq 2$. For any fixed $1\leq k_0 < 2M$, consider the set
\begin{eqnarray*}
\mathcal{S}(k_0):=\left\{1\leq k \leq \frac{M}{2} : k^2\equiv k_0\;(\mbox{mod}\;2M)\right\}.
\end{eqnarray*}
Let $\mathcal{M}_{f}(k_0) \subseteq \mathcal{S}(k_0)$ be non-empty and such that  $f(j)=0$ whenever $j\not\in \mathcal{M}_{f}(k_0)\cup\{M-k : k\in\mathcal{M}_{f}(k_0)\}$. Clearly, $S_f(k_0):=\mathcal{M}_{f}(k_0)\cup\{M-k : k\in\mathcal{M}_{f}(k_0)\}$ is the support of $f$. Consider the following partial theta series
\begin{equation}\label{thetaf0}
\theta_f(z):=\sum_{n\geq 0}f(n)\;q^{\frac{n^2}{2M}},\hspace{0.5cm}
\Theta_f(z):=\sum_{n\geq 0}nf(n)\;q^{\frac{n^2}{2M}}
\end{equation}
where $q=e^{2\pi i z}$, $z\in\mathbb{H}$. For $N\in\mathbb{N}$, let
\begin{eqnarray*}
\Gamma_1(N):=\left\{\begin{pmatrix}
a & b\\c & d\end{pmatrix}\in SL_2(\mathbb{Z}): c\equiv 0\;(\mbox{mod}\;N),\;a\equiv d\equiv 1\;(\mbox{mod}\;N)\right\}
\end{eqnarray*}
and let $\Gamma_M$ be defined as $\Gamma_1(2M)$ if $M$ is even and
\begin{eqnarray} \label{modd}
\left\{\begin{pmatrix}
a & b\\c & d\end{pmatrix}\in \Gamma_1(2M): b\equiv 0\;(\mbox{mod}\;2)\right\}
\end{eqnarray}
if $M$ is odd. Consider the set
\begin{equation} \label{bmset}
B_M := \{\alpha\in\mathbb{Q}: \alpha\;\mbox{is}\;\Gamma_M\mbox{-equivalent to}\;i\infty\}
\end{equation}
and let $A_M=A_{M,f}$ be defined by (\ref{bmset}) if $\sideset{}{'}\sum_{k\in \mathcal{M}_{f}(k_0)}f(k)\neq 0$ and by
\begin{equation*}
\{\alpha\in\mathbb{Q}: \alpha\;\mbox{is}\;\Gamma_M\mbox{-equivalent to}\;0\;\mbox{or}\;i\infty\}
\end{equation*} 
otherwise. Here and throughout, $\sideset{}{'}\sum$ means that whenever $\frac{M}{2} \in\mathcal{M}_{f}(k_0)$, we replace $f\bigl(\frac{M}{2}\bigr)$ in the sum by $\frac{1}{2} f\bigl(\frac{M}{2}\bigr)$. We also employ the convention that $f(n)=0$ if $n \not\in \mathbb{Z}$. For $k\in\frac{1}{2}\mathbb{Z}$ and $\gamma=\begin{pmatrix}
a & b\\c & d\end{pmatrix}\in\Gamma_M$, we define the Petersson slash operator $|_{k,\chi}$ by
\begin{equation*}
(g|_{k,\chi}\gamma)(\tau):=\overline{\chi(\gamma)}(c\tau+d)^{-k}g\left(\dfrac{a\tau+b}{c\tau+d}\right)
\end{equation*}
where $\tau \in \mathbb{C}$ and $\chi$ is a multiplier. Finally, we write $\left(\frac{\cdot}{\cdot}\right)$ for the extended Jacobi symbol and let $\varepsilon_d=1$ or $i$ according as $d\equiv 1$ or $3 \pmod{4}$. Our main result is now as follows. 

\begin{theorem} \label{main}
Let $f$ be a function with period $M \geq 2$ and support $S_f(k_0)$. Let $\alpha\in\mathbb{Q}$. If $f$ is even, then $\Theta_f(\alpha)$ is a quantum modular form of weight $3/2$ on $A_M$ with respect to $\Gamma_M$. If $f$ is odd, then $\theta_f(\alpha)$ is a ``strong" quantum modular form of weight $1/2$ on $\mathbb{Q}$ with respect to $\Gamma_M$ and is a quantum modular form of weight $1/2$ on $B_M$ with respect to $\Gamma_M$. 
\end{theorem}

\begin{remark} (i) The main novelty of Theorem \ref{main} is that one does not require $\Theta_f(z)$ or $\theta_f(z)$ to be a cusp form. For example, consider $\theta_{\psi}(z)$ where $\psi$ is given in Section 4.2 (cf. \cite{kll}). Otherwise, one can invoke (\ref{thetaf00}), (\ref{oddf}) and Theorem 1.1 in \cite{BR}.

(ii) In Theorem \ref{main}, $\Theta_f(z)$ satisfies 
\begin{eqnarray*}
\Theta_f(\alpha)-(\Theta_f|_{\frac{3}{2},\chi}\gamma)(\alpha)=r_{\gamma,f}(\alpha)
\end{eqnarray*}
for all $\gamma=\begin{pmatrix}
a & b\\c & d\end{pmatrix}\in\Gamma_M$ and $\alpha\in A_M$, where 
\begin{eqnarray*}
r_{\gamma,f}(z)=-\dfrac{\sqrt{M}\cdot e^{\frac{\pi i}{4}}}{2\pi}\int_{\gamma^{-1}(i\infty)}^{i\infty}\theta_f(\tau)(\tau-\bar{z})^{-\frac{3}{2}} \;d\tau.
\end{eqnarray*}
Here, $r_{\gamma,f}:\mathbb{R}\rightarrow\mathbb{C}$ is a $C^{\infty}$ function which is real-analytic in $\mathbb{R}\setminus\{\gamma^{-1}(i\infty)\}$ and $\chi$ is a multiplier given by
\begin{eqnarray} \label{multi}
\chi(\gamma)=e^{\frac{\pi i ab k_0}{M}} \left(\dfrac{2cM}{d}\right)\varepsilon_d^{-1}.
\end{eqnarray} 

(iii) For $\tau\in\mathbb{H}_{-} :=\{\tau\in\mathbb{C}:\text{Im}(\tau)<0\}$, let $\hat{\Theta}_f(\tau)$ denote the non-holomorphic Eichler integral 
\begin{eqnarray} \label{nonhol}
\hat{\Theta}_f(\tau):=\dfrac{1}{\sqrt{iM}} \int_{\bar{\tau}}^{i\infty} \Theta_f(w)(w-\tau)^{-\frac{1}{2}} \;dw.
\end{eqnarray}
In Theorem \ref{main}, $\theta_f(\alpha)$ is a ``strong" quantum modular form in the following sense (see \cite{LZ} or \cite{z}):
\begin{enumerate}
    \item $\theta_f$ and $\hat{\Theta}_f$ ``agree to infinite order" at all rational numbers (see Lemma \ref{LZsense}),
    \item for $\tau\in\mathbb{H}_{-}$ and $\gamma\in\Gamma_M$, we have
\begin{eqnarray*}
\hat{\Theta}_f(\tau)-(\hat{\Theta}_f|_{\frac{1}{2},\chi}\gamma)(\tau)=r_{\gamma,f}(\tau)
\end{eqnarray*}
where 
\begin{eqnarray*}
r_{\gamma,f}(\tau)=\dfrac{1}{\sqrt{iM}}\int_{\gamma^{-1}(i\infty)}^{i\infty}\Theta_f(w)(w-\tau)^{-\frac{1}{2}}\;dw.
\end{eqnarray*}
\end{enumerate}
Here, $r_{\gamma,f}(\tau)$ is a holomorphic function in $\mathbb{H}_{-}$, extends as a $\mathbb{C}^{\infty}$ function to $\mathbb{R}$ and is real-analytic in $\mathbb{R}\setminus\{\gamma^{-1}(i\infty)\}$. Also, $\chi$ is the multiplier as in (\ref{multi}). A close inspection of the techniques in \cite{LZ} reveals that one needs convergence of $\hat{\Theta}_f(\tau)$ for $\tau \in \mathbb{H}_{-}$ (and not necessarily at rational points) to deduce the strong quantum modularity property for $\theta_f(z)$. To ensure this condition, $\Theta_f(z)$ does not have to be a cusp form. For a similar approach, see \cite{bcr} and \cite{hikami3}.

(iv) If $f$ is a function with period $M \geq 2$ and support $S_f(k_0)$, then $\theta_f(z)$ is a sum of a modular form and a (strong) quantum modular form both of weight $1/2$ and $\Theta_f(z)$ is a sum of a modular form and a quantum modular form both of weight $3/2$. To see this, write $f$ as 
\begin{eqnarray*}
f(n)=f_e(n)+f_o(n)
\end{eqnarray*}
where 
\begin{eqnarray*}
f_e(n):=\dfrac{f(n)+f(-n)}{2},\;\;f_o(n):=\dfrac{f(n)-f(-n)}{2}.
\end{eqnarray*}
Clearly, $f_e(n)$ (respectively, $f_o(n)$) is an even (respectively, odd) function of period $M$ with support contained in $S_f(k_0)$. Indeed, if $S_{f,e}(k_0)$ denotes (respectively, $S_{f,o}(k_0)$) the support of $f_e(n)$ (respectively, $f_o(n)$), then $S_{f,e}(k_0)=\mathcal{M}_{f, e}(k_0)\cup\{M-k : k\in \mathcal{M}_{f, e}(k_0)\}$ (respectively, $S_{f,o}(k_0)=\mathcal{M}_{f, o}(k_0)\cup\{M-k : k\in \mathcal{M}_{f, o}(k_0)\}$) for some $\mathcal{M}_{f, e}(k_0), \mathcal{M}_{f, o}(k_0) \subseteq \mathcal{M}_{f}(k_0)$. Thus, we have
\begin{equation*}
\theta_f(z)=\sum_{n\geq 0}f_e(n)\;q^{\frac{n^2}{2M}}+\sum_{n\geq 0}f_o(n)\;q^{\frac{n^2}{2M}}=:\theta_{f}^{(e)}(z)+\theta_{f}^{(o)}(z),
\end{equation*}
\begin{equation*}
\Theta_f(z)=\sum_{n\geq 0}n\;f_e(n)\;q^{\frac{n^2}{2M}}+\sum_{n\geq 0}n\;f_o(n)\;q^{\frac{n^2}{2M}}=:\Theta_{f}^{(e)}(z)+\Theta_{f}^{(o)}(z).
\end{equation*} 
Now, apply Lemma \ref{trantheta} to $\theta_f^{(e)}(z)$ and Theorem \ref{main} to $\theta_{f}^{(o)}(z)$. Similarly, apply Theorem \ref{main} to $\Theta_f^{(e)}(z)$ and Lemma \ref{tranTheta} to $\Theta_f^{(o)}(z)$. 

(v) As pointed out by the referee, there is a ``duality" in the proof of Theorem \ref{main}. If $f$ is even, then the quantum modularity of $\Theta_f(z)$ is driven by the modularity of $\theta_f(z)$ and if $f$ is odd, then the (strong) quantum modularity of $\theta_f(z)$ is driven by the modularity of $\Theta_f(z)$. See Lemmas \ref{trantheta} and \ref{tranTheta}.

\end{remark}

The paper is organized as follows. In Section 2, we carefully study some important transformation and limiting properties of $\theta_f(z)$ and $\Theta_f(z)$. In Section 3, we prove Theorem \ref{main}. In Section 4, we give some examples, including the quantum modularity of the Kontsevich-Zagier series $\mathscr{F}_t(q)$ associated to the family of torus knots $T(3,2^t)$, $t \geq 2$. This latter result generalizes the quantum modularity of $F(q)$.

\section{Preliminaries}

We begin with transformation properties of the partial theta series $\theta_f(z)$ and $\Theta_f(z)$ in (\ref{thetaf0}).

\begin{lemma}\label{trantheta}
Let $f$ be an even function with period $M\geq 2$ and support $S_f(k_0)$. For all $\gamma=\begin{pmatrix}
a & b\\c & d\end{pmatrix}\in\Gamma_M$, we have
\begin{equation} \label{trans1}
\theta_f(\gamma z)=e^{\frac{\pi iabk_0}{M}} \left(\dfrac{2cM}{d}\right)\varepsilon_d^{-1}(cz+d)^{\frac{1}{2}}\theta_f(z).
\end{equation}
\end{lemma}

\begin{proof}
From \eqref{thetaf0}, we have
\begin{eqnarray} 
\theta_f(z)&=&\sum_{1\leq k< M}\sum_{n=0}^\infty f(Mn+k)\;q^{\frac{(Mn+k)^2}{2M}}\notag\\
&=&\sum_{1\leq k<M}f(k)\sum_{n\geq 0}q^{\frac{(Mn+k)^2}{2M}}\notag\\
&=&\sum_{1\leq k< \frac{M}{2}}f(k)\left(\sum_{n\geq 0}q^{\frac{(Mn+k)^2}{2M}}+\sum_{n\geq 0}q^{\frac{(Mn+M-k)^2}{2M}}\right)+ f\Bigl(\frac{M}{2}\Bigr) \;\sum_{n\geq 0}q^{\frac{\bigl(Mn+ \frac{M}{2} \bigr)^2}{2M}} \label{b4} \\
&=&\sum_{1\leq k< \frac{M}{2}}f(k)\sum_{n=-\infty}^\infty q^{\frac{(Mn+k)^2}{2M}}+\delta_f\Bigl(\frac{M}{2} \Bigr)\;\sum_{n=-\infty}^\infty q^{\frac{\bigl(Mn+\frac{M}{2}\bigr)^2}{2M}} \label{after}
\end{eqnarray}
where 
\begin{equation*}
\delta_f\Bigl(\frac{M}{2}\Bigr) :=
\begin{cases}
\frac{1}{2} f\bigl(\frac{M}{2}\bigr) & \text{if $M$ is even and $\frac{M}{2} \in\mathcal{M}_{f}(k_0)$,} \\
0 & \text{otherwise.}
\end{cases}
\end{equation*}
Note that \eqref{after} follows by changing $n\rightarrow -n-1$ in the second sum in (\ref{b4}).
Since support of $f$ is $S_f(k_0)$, \eqref{after} yields
\begin{eqnarray}\label{thetaf00}
\theta_f(z)=\sideset{}{'}\sum_{k \in \mathcal{M}_{f}(k_0)}f(k)\;\theta(z;k,M)
\end{eqnarray}
where $\theta(z;k,M)$ is the theta series
\begin{eqnarray*}\label{theta}
\theta(z;k,M):=\sum_{n=-\infty}^\infty q^{\frac{(Mn+k)^2}{2M}}.
\end{eqnarray*}
By Proposition 2.1 in \cite{shimura}, we see that $\theta(z;k,M)$ satisfies
\begin{eqnarray}\label{thetatrans1}
\theta(\gamma z;k,M)=e^{\frac{ \pi i ab k^2}{M}} \left(\dfrac{2cM}{d}\right)\varepsilon_d^{-1}(cz+d)^{\frac{1}{2}}\theta(z;ak,M)
\end{eqnarray}
for all $\gamma=\begin{pmatrix}
a & b\\c & d\end{pmatrix}\in\Gamma_M$. Also, since $\gamma\in\Gamma_M$, we have for some integer $j$ that
\begin{eqnarray} \label{thetakMtrans1}
\theta(z;ak,M)=\sum_{n=-\infty}^\infty q^{\frac{(Mn+ak)^2}{2M}}
=\sum_{n=-\infty}^\infty q^{\frac{(Mn+(1+2jM)k)^2}{2M}}=\theta(z;k,M)
\end{eqnarray}
where $n$ has been replaced by $n-2jk$ in the second sum in (\ref{thetakMtrans1}). Noting that 
\begin{eqnarray}\label{expo1}
e^{\frac{\pi i ab k^2}{M}}=e^{\frac{\pi i ab k_0}{M}},
\end{eqnarray}
(\ref{trans1}) now follows from \eqref{thetaf00} and \eqref{thetatrans1}--\eqref{expo1}.
\end{proof}

\begin{lemma}\label{tranTheta}
Let $f$ be an odd function with period $M\geq 2$ and support $S_f(k_0)$. For all $\gamma=\begin{pmatrix}
a & b\\c & d\end{pmatrix}\in\Gamma_M$, we have
\begin{equation} \label{trans2}
\Theta_f(\gamma z)=e^{\frac{\pi iabk_0}{M}} \left(\dfrac{2cM}{d}\right)\varepsilon_d^{-1}(cz+d)^{\frac{3}{2}}\Theta_f(z).
\end{equation}
\end{lemma}

\begin{proof}
If $M$ is even, then $f\bigl(\frac{M}{2} \bigr)=0$ for odd $f$. So, we have
\begin{eqnarray}\label{oddf}
\Theta_f(z)&=&\sum_{0\leq k<M}\sum_{n\geq 0}(Mn+k)\;f(Mn+k)\;q^{\frac{(Mn+k)^2}{2M}}\notag\\
&=&\sum_{0\leq k\leq \frac{M}{2}}f(k)\left(\sum_{n\geq 0}(Mn+k)\;q^{\frac{(Mn+k)^2}{2M}}-\sum_{n\geq 0}(Mn+(M-k))\;q^{\frac{(Mn+M-k)^2}{2M}}\right)\notag\\
&=&\sum_{0\leq k\leq \frac{M}{2}}f(k)\sum_{n=-\infty}^\infty (Mn+k)\;q^{\frac{(Mn+k)^2}{2M}}\notag\\
&=&\sum_{k\in\mathcal{M}_{f}(k_0)}f(k)\;\tilde{\Theta}(z;k,M)
\end{eqnarray}
where 
\begin{eqnarray*}
\tilde{\Theta}(z;k,M)=\sum_{n=-\infty}^\infty(Mn+k)\;q^{\frac{(Mn+k)^2}{2M}}.
\end{eqnarray*}
Using \cite[Proposition 2.1]{shimura} (with $A=[M]$, $\nu=1$ and $P(m)=m$), we have 
\begin{eqnarray}\label{thetatrans4}
\tilde{\Theta}(\gamma z;k,M)=e^{\frac{\pi iabk^2}{M}}\left(\dfrac{2cM}{d}\right)\varepsilon_d^{-1}(cz+d)^{\frac{3}{2}}\;\tilde{\Theta}(z;ak,M)
\end{eqnarray}
for all $\gamma=\begin{pmatrix}
a & b\\c & d\end{pmatrix}\in\Gamma_M$. Also, since $\gamma\in\Gamma_M$, we have for some integer $j$ that
\begin{eqnarray}\label{thetakMtrans}
\tilde{\Theta}(z;ak,M)&=&\sum_{n=-\infty}^\infty(Mn+ak)\; q^{\frac{(Mn+ak)^2}{2M}}\notag\\
&=&\sum_{n=-\infty}^\infty(Mn+(1+2jM)k)\; q^{\frac{(Mn+(1+2jM)k)^2}{2M}}=\tilde{\Theta}(z;k,M)
\end{eqnarray}
where $n$ has been replaced by $n-2jk$ in the sum in (\ref{thetakMtrans}). Thus, combining \eqref{oddf}--\eqref{thetakMtrans} yields (\ref{trans2}).
\end{proof}

\begin{lemma}\label{infexpo}
Let $f$ be an even function with period $M\geq 2$ and support $S_f(k_0)$. Then
\begin{eqnarray} \label{linearcombprod}
\theta_f(z)=q^{\frac{k_0}{2M}}(q^M;q^M)_\infty\sideset{}{'}\sum_{k\in \mathcal{M}_{f}(k_0)}q^{k'}f(k)\;(-q^{\frac{M}{2}-k};q^M)_\infty(-q^{\frac{M}{2}+k};q^M)_\infty
\end{eqnarray}
where for $k\in \mathcal{M}_{f}(k_0)$, $k'=\dfrac{k^2-k_0}{2M}\in\mathbb{Z}_{\geq 0}$.
\end{lemma}

\begin{proof}
We have
\begin{eqnarray}\label{thetajac}
\theta(z;k,M)&=&q^{\frac{k^2}{2M}}\sum_{n=-\infty}^\infty q^{\frac{M^2n^2+2kMn}{2M}}=q^{\frac{k^2}{2M}}\sum_{n=-\infty}^\infty (q^k)^n\;\left(q^{\frac{M}{2}}\right)^{n^2}\notag\\
&=&q^{\frac{k^2}{2M}}(q^M;q^M)_\infty(-q^{\frac{M}{2}-k};q^M)_\infty(-q^{\frac{M}{2}+k};q^M)_\infty
\end{eqnarray}
where (\ref{thetajac}) follows from Jacobi's triple product identity 
\begin{eqnarray} \label{jtp}
\sum_{n=-\infty}^\infty (-1)^nz^{n}q^{n^2}=(q^2;q^2)_\infty(zq;q^2)_\infty(z^{-1}q;q^2)_\infty
\end{eqnarray}
with $z\rightarrow -q^{k}$ and $q\rightarrow q^{\frac{M}{2}}$ in (\ref{jtp}). As in the proof of Lemma \ref{trantheta}, we have
\begin{eqnarray}\label{thetaM0}
\theta_f(z)=\sideset{}{'}\sum_{k\in \mathcal{M}_{f}(k_0)}f(k)\;\theta(z;k,M).
\end{eqnarray}
Thus, \eqref{thetajac} and \eqref{thetaM0} imply (\ref{linearcombprod}) where for $k\in \mathcal{M}_{f}(k_0)$, $k'=\dfrac{k^2-k_0}{2M}$. Since $k'\in\mathbb{Z}$ implies $k^2\geq k_0$, we conclude that $k'\in\mathbb{Z}_{\geq 0}$. 
\end{proof}

\begin{lemma}\label{expodecay0}
Let $f$ be an even function with period $M\geq 2$ and support $S_f(k_0)$. Assume $\sideset{}{'}\sum_{k\in\mathcal{M}_{f}(k_0)}f(k)=0$. Let $\alpha\in\mathbb{Q}$ be such that $\alpha\neq \gamma(i\infty)$ for any $\gamma \in \Gamma_M$. Then we have
\begin{eqnarray} \label{trans3}
\theta_f(\alpha+iy)=\dfrac{1}{\sqrt{M (y-i\alpha)}}\sideset{}{'}\sum_{k\in\mathcal{M}_{f}(k_0)}f(k)\sum_{\substack{n=-\infty\\n\neq 0}}^\infty e^{-\frac{\pi n^2}{M(y-i\alpha)}+\frac{2\pi ink}{M}}.
\end{eqnarray}
\end{lemma}

\begin{proof}
For any $1\leq k<M$, we obtain the following upon using \cite[Chapter 5, pp. 76]{Rad} with $u=(y-i\alpha)M$ and $x=\frac{k}{M}$:
\begin{eqnarray}\label{thetaf00trans}
\theta(\alpha+iy;k,M)&=&\dfrac{e^{\frac{\pi i(\alpha+iy)k^2+\pi(y-i\alpha)k^2}{M}}}{\sqrt{M (y-i\alpha)}}\sum_{n=-\infty}^\infty e^{-\frac{\pi n^2}{M (y-i\alpha)}+\frac{2\pi ink}{M}}\notag\\
&=&\dfrac{1}{\sqrt{M (y-i\alpha)}}\left(1+\sum_{\substack{n=-\infty\\n\neq 0}}^\infty e^{-\frac{\pi n^2}{M (y-i\alpha)}+\frac{2\pi ink}{M}}\right).
\end{eqnarray}
As in the proof of Lemma \ref{trantheta}, we have
\begin{eqnarray}\label{thetaM1}
\theta_f(z)=\sideset{}{'}\sum_{k\in \mathcal{M}_{f}(k_0)}f(k)\;\theta(z;k,M)
\end{eqnarray}
and so (\ref{trans3}) follows from \eqref{thetaf00trans}, \eqref{thetaM1} and $\sideset{}{'}\sum_{k\in\mathcal{M}_{f}(k_0)}f(k)=0$. \end{proof}

\begin{corollary}\label{expodecay00}
Let $f$ be an even function with period $M\geq 2$ and support $S_f(k_0)$. Assume $\sideset{}{'}\sum_{k\in\mathcal{M}(k_0)}f(k)=0$. Then we have
\begin{eqnarray*}
\theta_f(iy)=\dfrac{e^{-\frac{\pi}{M y}}}{\sqrt{M y}}\left(c_f(M,k_0)+o(1)\right)
\end{eqnarray*}
where $o(1)\rightarrow 0$ as $y\rightarrow 0^+$ and
\begin{eqnarray*}
c_f(M,k_0):=2\sideset{}{'}\sum_{k\in\mathcal{M}_{f}(k_0)}f(k)\cos\left(\frac{2\pi k}{M}\right).
\end{eqnarray*}
\end{corollary}

\begin{proof}
First, we note that $\gamma(i\infty)\neq 0$ for all $\gamma\in\Gamma_M$. Thus, we choose $\alpha=0$ in Lemma \ref{expodecay0} to get
\begin{eqnarray}\label{expodecay000}
\theta_f(iy)&=&\dfrac{1}{\sqrt{My}}\sideset{}{'}\sum_{k\in\mathcal{M}_{f}(k_0)}f(k)\sum_{\substack{n=-\infty\\n\neq 0}}^\infty e^{-\frac{\pi n^2}{My}+\frac{2\pi ink}{M}}\notag\\
&=&\dfrac{1}{\sqrt{My}}\sideset{}{'}\sum_{k\in\mathcal{M}_{f}(k_0)}f(k)\sum_{n=1}^\infty(r_M(k,n)+r_M(k,-n))\;q_1^{n^2}
\end{eqnarray}
where $q_1=e^{-\frac{\pi}{My}}$ and $r_M(k,n)=e^{\frac{2\pi ink}{M}}$. Interchanging the sums in \eqref{expodecay000}, we find
\begin{eqnarray*}
\theta_f(iy)=\dfrac{1}{\sqrt{My}}\sum_{n=1}^\infty g_M(k_0,n)\;q_1^{n^2}
\end{eqnarray*}
where
\begin{eqnarray*}
g_M(k_0,n):=\sideset{}{'}\sum_{k\in\mathcal{M}_{f}(k_0)}f(k)(r_M(k,n)+r_M(k,-n)).
\end{eqnarray*}
At this point, note that $g_M(k_0,n)=O(1)$ and $q_1\rightarrow 0$ as $y\rightarrow 0^+$. This yields the result.
\end{proof}

Let $C: \mathbb{Z} \rightarrow \mathbb{C}$ be a periodic function with mean value zero and consider the $L$-series
\begin{eqnarray*}
L(s,C) :=\sum_{n=1}^\infty \dfrac{C(n)}{n^s},\; \Re(s) > 0,
\end{eqnarray*}
which has an analytic continuation to $\mathbb{C}$ \cite[Proposition, page 98]{LZ}.

\begin{lemma}\label{LZsense}
Let $f$ be an odd function with period $M\geq 2$ and support $S_f(k_0)$. Then as $t\rightarrow 0^+$, we have for $(p,q)=1$ that
\begin{eqnarray}
\theta_f\left(\dfrac{p}{q}+\dfrac{it}{2\pi}\right)&\sim &\sum_{r=0}^\infty L(-2r,C_{f,k_0})\dfrac{\left(-\frac{t}{2M}\right)^r}{r!},\label{thetaration}\\
\hat{\Theta}_f\left(\dfrac{p}{q}-\dfrac{it}{2\pi}\right)&\sim &\sum_{r=0}^\infty L(-2r,C_{f,k_0})\dfrac{\left(\frac{t}{2M}\right)^r}{r!}\label{theta1ration}
\end{eqnarray}
where $C_{f,k_0}(n):=\sum_{k\in\mathcal{M}_{f}(k_0)}f(k)\;C_f(n,k)$ and 
\begin{equation*}
C_f(n,k):=\begin{cases}
    e^{\frac{\pi i pn^2}{Mq}} & \emph{if}\;n\equiv k\;(\emph{mod}\;M),\\
     -e^{\frac{\pi i pn^2}{Mq}} & \emph{if}\;n\equiv -k\;(\emph{mod}\;M),\\
     0 &\emph{otherwise}.
\end{cases}
\end{equation*}
Thus, in the sense of Lawrence and Zagier \cite[page 103]{LZ}, \eqref{thetaration} and \eqref{theta1ration} imply that $\theta_f$ and $\hat{\Theta}_f$ ``agree to infinite order" at all rational numbers.
\end{lemma}

\begin{proof}
For $t>0$, we have
\begin{eqnarray}\label{thetaL}
\theta_f\left(\dfrac{p}{q}+\dfrac{it}{2\pi}\right)&=&\sum_{n\geq 0}f(n)\;e^{\frac{\pi i pn^2}{Mq}-\frac{tn^2}{2M}}\notag\\
&=&\sum_{k\in\mathcal{M}_{f}(k_0)}f(k)\left(\sum_{\substack{n>0\\n\equiv k\;(\text{mod}\;M)}}e^{\frac{\pi i pn^2}{Mq}-\frac{tn^2}{2M}}-\sum_{\substack{n>0\\n\equiv -k\;(\text{mod}\;M)}}e^{\frac{\pi i pn^2}{Mq}-\frac{tn^2}{2M}}\right)\notag\\
&=&\sum_{n\geq 1}C_{f,k_0}(n)\;e^{-\frac{tn^2}{2M}}.
\end{eqnarray}
For each $k\in\mathcal{M}_{f}(k_0)$, $C_{f}(n,k)$ is an odd function with period $Mq$ or $2Mq$ according as $2$ divides $p$ or does not divide $p$. Also, $C_f(n,k)$ has mean value zero. This implies that $C_{f,k_0}(n)$ is an odd function with period $Mq$ or $2Mq$ according as $2$ divides $p$ or does not divide $p$ and with mean value zero. Thus, by \cite[Proposition, page 98]{LZ} and \eqref{thetaL}, we obtain
\begin{eqnarray*}
\theta_f\left(\dfrac{p}{q}+\dfrac{it}{2\pi}\right)\sim\sum_{r=0}^\infty L(-2r,C_{f,k_0})\dfrac{\left(-\frac{t}{2M}\right)^r}{r!}.
\end{eqnarray*}
Next, we turn to $\hat{\Theta}_f(\tau)$ where $\tau=x+iy$ with $y<0$. First, we have for $w \in \mathbb{H}$ 
\begin{eqnarray*}
\Theta_f(w)=\sum_{k\in\mathcal{M}_{f}(k_0)}f(k)\left(\sum_{\substack{n>0\\n\equiv k\;(\text{mod}\;M)}}n\;e^{\frac{\pi in^2w}{M}}-\sum_{\substack{n>0\\n\equiv -k\;(\text{mod}\;M)}}n\;e^{\frac{\pi in^2w}{M}}\right).
\end{eqnarray*}
Thus, by the change of variable $w\rightarrow w+\tau$ and contour integration, it follows that
\begin{eqnarray}\label{Eiceval1}
\hat{\Theta}_f(\tau)=\dfrac{1}{\sqrt{iM}}\sum_{k\in\mathcal{M}_{f}(k_0)}f(k)\left(\sum_{\substack{n>0\\n\equiv k\;(\text{mod}\;M)}}-\sum_{\substack{n>0\\n\equiv -k\;(\text{mod}\;M)}}\right)n e^{\frac{\pi i n^2\tau}{M}}\int_{-2iy}^{i\infty}e^{\frac{\pi in^2 w}{M}}w^{-\frac{1}{2}}dw.
\end{eqnarray}
To evaluate the integral on the right-hand side of \eqref{Eiceval1}, we let $w\rightarrow \frac{iMw}{\pi n^2}$. This yields
\begin{eqnarray}\label{evalcv}
\int_{-2iy}^{i\infty}e^{\frac{\pi in^2 w}{M}}w^{-\frac{1}{2}}dw=\sqrt{\dfrac{iM}{\pi n^2}}\int_{-\frac{2\pi yn^2}{M}}^\infty e^{-w}w^{-\frac{1}{2}}dw=\sqrt{\dfrac{iM}{\pi n^2}}\;\Gamma\left(\dfrac{1}{2},-\dfrac{2\pi yn^2}{M}\right)
\end{eqnarray}
where $\Gamma(a,x)$ is the upper incomplete gamma function defined by
\begin{eqnarray*}
\Gamma(a,x):=\int_x^{\infty}w^{a-1}e^{-w}dw.
\end{eqnarray*}
Thus, \eqref{Eiceval1} and \eqref{evalcv} yield
\begin{eqnarray}\label{evalcv2}
\hat{\Theta}_f(\tau)=\dfrac{1}{\sqrt{\pi}}\sum_{k\in\mathcal{M}_{f}(k_0)}f(k)\left(\sum_{\substack{n>0\\n\equiv k\;(\text{mod}\;M)}}-\sum_{\substack{n>0\\n\equiv -k\;(\text{mod}\;M)}}\right)e^{\frac{\pi i n^2\tau}{M}}\Gamma\left(\dfrac{1}{2},-\dfrac{2\pi yn^2}{M}\right)
\end{eqnarray}
and so \eqref{evalcv2} implies
\begin{eqnarray}\label{evalrati}
\hat{\Theta}_f\left(\dfrac{p}{q}-\dfrac{it}{2\pi}\right)=\dfrac{1}{\sqrt{\pi}}\sum_{n\geq 1}\;C_{f,k_0}(n)\;e^{\frac{tn^2}{2M}}\;\Gamma\left(\dfrac{1}{2},\dfrac{tn^2}{M}\right).
\end{eqnarray}
Since $C_{f,k_0}(n)$ is an odd periodic function, it now follows from \cite[Lemma 4.3]{BM} and \eqref{evalrati} that
\begin{eqnarray*}
\hat{\Theta}_f\left(\dfrac{p}{q}-\dfrac{it}{2\pi}\right)\sim \sum_{r=0}^\infty L(-2r,C_{f,k_0})\dfrac{\left(\frac{t}{2M}\right)^r}{r!}.
\end{eqnarray*}
\end{proof}

\section{Proof of Theorem \ref{main}}
We are now in a position to prove Theorem \ref{main}.
\begin{proof}[Proof of Theorem \ref{main}]
Let $f$ be an even function with period $M \geq 2$ and support $S_f(k_0)$. Define the Eichler integral of $\theta_f(z)$ as follows:
\begin{eqnarray*}
\tilde{\theta}_f(z):=\int_{z}^{i\infty}\theta_f(\tau)(\tau-\bar{z})^{-\frac{3}{2}}\;d\tau.
\end{eqnarray*}
In view of Lemma \ref{infexpo} and Corollary \ref{expodecay00}, we see that $\tilde{\theta}_f(z)$ is well-defined on $\mathbb{H}\cup A_M$. Thus, for $z=x+iy\in\mathbb{H}\cup A_M$, it follows using contour integration that
\begin{eqnarray}\label{contint}
\tilde{\theta}_f(z)=\sqrt{\dfrac{\pi}{M}}e^{-\frac{i\pi}{4}}\sum_{n\geq 0}n\;f(n)\;\Gamma\left(-\dfrac{1}{2},\dfrac{2\pi n^2 y}{M}\right)e^{\frac{\pi i n^2}{M}\bar{z}}.
\end{eqnarray}
For $\alpha\in A_M$, we see from \eqref{contint} and the fact that $\Gamma(-\frac{1}{2})=-2\sqrt{\pi}$
\begin{eqnarray}\label{connect}
\tilde{\theta}_f(\alpha)=-\dfrac{2\pi}{\sqrt{M}}e^{-\frac{i\pi}{4}}\Theta_f(\alpha).
\end{eqnarray}
For $\gamma=\begin{pmatrix}
a & b\\c & d\end{pmatrix}\in\Gamma_M$, it follows from Lemma \ref{trantheta}, \eqref{connect} and
\begin{equation*} \label{transform1}
d(\gamma\tau)=\dfrac{d\tau}{(c\tau+d)^2},\hspace{0.5cm}\gamma\tau-\gamma \bar{z}=\dfrac{\tau-\bar{z}}{(c\tau+d)(c\bar{z}+d)}
\end{equation*}
that
\begin{eqnarray*}\label{cocycle}
\Theta_f(\alpha)-(\Theta_f|_{\frac{3}{2},\chi}\gamma)(\alpha)=:r_{\gamma,f}(\alpha)
\end{eqnarray*}
where
\begin{eqnarray*}\label{rgamma}
r_{\gamma,f}(z)=-\dfrac{\sqrt{M}\cdot e^{\frac{i\pi}{4}}}{2\pi}\int_{\gamma^{-1}(i\infty)}^{i\infty}\theta_f(\tau)(\tau-\bar{z})^{-\frac{3}{2}}\;d\tau.
\end{eqnarray*}
It only remains to observe that $r_{\gamma,f}(z)$ is $C^{\infty}$ and real-analytic in $\mathbb{R}\setminus\{\gamma^{-1}(i\infty)\}$.

Let $f$ be an odd function with period $M \geq 2$ and support $S_f(k_0)$. For $\tau\in\mathbb{H}_{-}$ and $\gamma\in\Gamma_M$, it follows from (\ref{nonhol}) and Lemma \ref{tranTheta} that
\begin{eqnarray}\label{transpropTheta}
\hat{\Theta}_f(\tau)-(\hat{\Theta}_f|_{\frac{1}{2},\chi}\gamma)(\tau)=r_{\gamma,f}(\tau)
\end{eqnarray}
where 
\begin{eqnarray*} \label{trans}
r_{\gamma,f}(\tau)=\dfrac{1}{\sqrt{iM}}\int_{\gamma^{-1}(i\infty)}^{i\infty} \Theta_f(w) (w-\tau)^{-\frac{1}{2}} \; dw.
\end{eqnarray*}
Since by Lemma \ref{LZsense}, $\theta_f$ and $\hat{\Theta}_f$ ``agree to infinite order" at all rational numbers and $\hat{\Theta}_f(\tau)$ satisfies the transformation property in \eqref{transpropTheta} for all $\tau\in\mathbb{H}_{-}$, it follows in the sense of Lawrence and Zagier \cite[Page 103]{LZ} that $\theta_f(z)$ is a strong quantum modular form of weight $1/2$ on $\mathbb{Q}$ with respect to $\Gamma_M$. It is also clear that $r_{\gamma,f}(\tau)$ is a holomorphic function in $\mathbb{H}_{-}$, extends as a $\mathbb{C}^{\infty}$ function to $\mathbb{R}$ and is real-analytic in $\mathbb{R}\setminus\{\gamma^{-1}(i\infty)\}$. Here, $\chi$ is the multiplier given by (\ref{multi}). Finally, we note that as $\Theta_f(z)$ vanishes at $z=i\infty$ (and thus at all $\alpha \in B_M$), $\hat{\Theta}_f$ is well-defined on $B_M$. Thus, for $\tau\in\mathbb{H}_{-}\cup B_M$, $\hat{\Theta}_f(\tau)$ satisfies (\ref{transpropTheta}). Now, one can check that
\begin{equation} \label{thetafsum}
\theta_{f}(z) = \sum_{k\in\mathcal{M}_f(k_0)}f(k)\left(\sum_{\substack{n\geq 0 \\ n\equiv k\;(\text{mod}\;M)}}-\sum_{\substack{n\geq 0\\ n\equiv -k\;(\text{mod}\;M) }}\right) q^{\frac{n^2}{2M}}. 
\end{equation}
For $\tau=\alpha \in B_M$, it follows from (\ref{evalcv2}), (\ref{thetafsum}) and $\Gamma(\frac{1}{2})=\sqrt{\pi}$ that $\theta_f(\alpha) = \hat{\Theta}_f(\alpha)$. Thus, $\theta_f(\alpha)$ is a quantum modular form of weight $1/2$ on $B_M$ with respect to $\Gamma_M$.

\end{proof}

\section{Examples}

In this section, we illustrate Theorem \ref{main} with four examples.

\subsection{Kontsevich-Zagier series $\mathscr{F}_t(q)$ for torus knots $T(3,2^t)$}
Let $K$ be a knot and $J_N(K;q)$ be the usual colored Jones polynomial, normalized to be $1$ for the unknot. For the importance of this quantum knot invariant, see, for example, \cite{bd}, \cite{g}, \cite{my} or \cite{z}. If $T(3,2)$ is the right-handed torus knot, then \cite{habiro, thang}
\begin{equation} \label{t32}
J_N(T(3,2); q) = q^{1-N} \sum_{n \geq 0} q^{-nN} (q^{1-N})_n.
\end{equation}
Upon comparing (\ref{kz}) and (\ref{t32}), we immediately observe that $F(q)$ matches the colored Jones polynomial for $T(3,2)$ at a root of unity $q=\zeta_N := e^{\frac{2\pi i}{N}}$, that is,
\begin{equation*}
\zeta_N F(\zeta_N) = J_N(T(3,2); \zeta_N).
\end{equation*}

Consider the family of torus knots $T(3, 2^t)$ for an integer $t \geq 2$. In this case, a $q$-hypergeometric expression for the colored Jones polynomial has been computed, namely (see page 41, Th{\'e}or{\`e}me 3.2 in \cite{konan}, cf. \cite{hk2})
\begin{align} \label{t32t}
J_N(T(3,2^t); q) & = (-1)^{h''(t)} q^{2^t - 1 - h'(t) - N} \sum_{n \geq 0} (q^{1-N})_n q^{-Nn m(t)} \nonumber \\
& \times \sum_{3 \sum_{\ell=1}^{m(t)-1} j_{\ell} \ell \,\equiv \,1\;(\tiny{\mbox{mod}} \; m(t))} (-q^{-N})^{\sum_{\ell=1}^{m(t)-1} j_{\ell}} q^{\frac{-a(t) + \sum_{\ell=1}^{m(t)-1} j_{\ell} \ell}{m(t)} + \sum_{\ell=1}^{m(t)-1} \binom{j_{\ell}}{2}} \nonumber \\
& \quad \quad \times \sum_{k=0}^{m(t) - 1} q^{-kN} \prod_{\ell=1}^{m(t) - 1} \begin{bmatrix} n + I(\ell \leq k) \\ j_{\ell} \end{bmatrix} 
\end{align}
where 
\begin{equation*}
h''(t) = 
\begin{cases}
\frac{2^t-1}{3} &\text{if $t$ is even}, \\
\frac{2^t -2}{3} &\text{if $t$ is odd}, \\
\end{cases} \quad \quad
h'(t) =
\begin{cases}
\frac{2^t-4}{3} &\text{if $t$ is even}, \\
\frac{2^t -5}{3} &\text{if $t$ is odd}, \\
\end{cases} \quad \quad 
a(t)=
\begin{cases}
\frac{2^{t-1} +1}{3} &\text{if $t$ is even}, \\
\frac{2^t +1}{3} &\text{if $t$ is odd}, \\
\end{cases}
\end{equation*}
$m(t)=2^{t-1}$, $I(*)$ is the characteristic function and
\begin{equation*}
\begin{bmatrix} n \\ k \end{bmatrix} = \begin{bmatrix} n \\ k \end{bmatrix}_{q} := \frac{(q)_n}{(q)_{n-k} (q)_k}
\end{equation*}
is the $q$-binomial coefficient. We now define the {\it Kontsevich-Zagier series for torus knots $T(3,2^t)$} as\footnote{For $t=1$, one may define the sum over the $j_{\ell}$ to be 1 in (\ref{t32t}) and (\ref{kztorus}) to recover (\ref{t32}) and (\ref{kz}).}
\begin{align} \label{kztorus} 
\mathscr{F}_t(q) & = (-1)^{h''(t)} q^{-h'(t)} \sum_{n \geq 0} (q)_n \sum_{3 \sum_{\ell = 1}^{m(t) - 1} j_{\ell} \ell \,\equiv \,1\;(\tiny{\mbox{mod}} \; m(t))} q^{\frac{-a(t) + \sum_{\ell=1}^{m(t) - 1} j_{\ell} \ell}{m(t)} + \sum_{\ell=1}^{m(t)-1} \binom{j_{\ell}}{2}} \nonumber \\
& \times \sum_{k=0}^{m(t)-1} \prod_{\ell=1}^{m(t) - 1} \begin{bmatrix} n + I(\ell \leq k) \\ j_{\ell} \end{bmatrix}.
\end{align}
The expression $\mathscr{F}_t(q)$ converges in a similar manner as $F(q)$ and, by (\ref{t32t}) and (\ref{kztorus}), satisfies

\begin{equation*}
\zeta_{N}^{2^t-1} \mathscr{F}_t(\zeta_N) = J_N(T(3,2^t); \zeta_N).
\end{equation*}

An application of Theorem \ref{main} is the following. For an integer $t \geq 2$, set $s_t := \frac{(2^{t+1} - 3)^2}{3\cdot 2^{t+2}}$.

\begin{corollary}
For an integer $t \geq 2$ and $\alpha \in\mathbb{Q}$, $\phi_t(\alpha) :=e^{2\pi i s_t \alpha}\mathscr{F}_t(e^{2\pi i\alpha})$ is a quantum modular form of weight $3/2$ on $A_{3\cdot 2^{t+1}} = \{\alpha\in\mathbb{Q}: \alpha\;\mbox{is}\;\Gamma_1(3\cdot 2^{t+2}) \mbox{-equivalent to}\;0\;\mbox{or}\;i\infty\}$ with respect to $\Gamma_1(3\cdot 2^{t+2})$. 
\end{corollary}

\begin{proof}
The Kontsevich-Zagier series $\mathscr{F}_t(q)$ satisfies the ``strange" identity (see Proposition 2.4 in \cite{Be})\footnote{Taking $t=1$ in (\ref{KZstrange}) and (\ref{chit}) recovers (\ref{strid}).}
\begin{eqnarray} \label{KZstrange}
\mathscr{F}_t(q)``=" -\dfrac{1}{2}\Theta_{\chi_t}(z)
\end{eqnarray}
where
\begin{equation} \label{chit}
\chi_t(n) :=
\begin{cases}
1 &\text{if $n \equiv 2^{t+1}-3$, $3+2^{t+2} \; (\text{mod}\; 3\cdot2^{t+1}),$} \\
-1 &\text{if $n \equiv 2^{t+1} + 3$, $2^{t+2} - 3 \; (\text{mod}\; 3\cdot2^{t+1}),$} \\
0 &\text{otherwise.}
\end{cases}
\end{equation} 
Note that $\chi_t$ is an even function with period $M=3\cdot 2^{t+1}$. For $k_0=(2^{t+1}-3)^2 \pmod{3\cdot 2^{t+2}}$, consider the set $\mathcal{M}_{\chi_t}(k_0)=\{2^{t+1}-3, 2^{t+1}+3\}$. Thus, $S_{\chi_t}(k_0)=\{\pm(2^{t+1}-3),\pm(2^{t+1}+3)\}$. By Theorem \ref{main} and (\ref{KZstrange}), the result follows.
\end{proof} 

\subsection{Generating function for odd balanced unimodal sequences}
Let $v(n)$ denote the number of odd-balanced unimodal sequences of weight $2n+2$ and $v(m,n)$ the number of such sequences having rank $m$. In \cite{kll}, the authors study the bivariate generating function 
\begin{equation*}
\mathcal{V}(x,q) := \sum_{n\geq 0}\dfrac{(-xq,-x/q)_nq^n}{(q,q^2)_{n+1}} = \sum_{\substack{n \geq 0 \\ m \in \mathbb{Z}}} \, v(m,n) x^m q^n
\end{equation*}
and prove that for $\alpha \in \mathbb{Q}$, $q^{-7}\mathcal{V}(-1,q^{-8})\big\rvert_{z\rightarrow \alpha}$ is a quantum modular form of weight $3/2$ on $A=\{\alpha\in\mathbb{Q}: \alpha\;\mbox{is}\;\Gamma_0(16)\mbox{-equivalent to}\;i\infty\}$ with respect to $\Gamma_0(16)$. A slight variant of this result is as follows. If we let $q\rightarrow q^{2}$ in the identity (see \cite[page 3693]{kll})
\begin{eqnarray*}\label{V}
\mathcal{V}(-1,q^{-1})=-\dfrac{q}{2}\sum_{n\geq 0}(2n+1)\;q^{n(n+1)/2},
\end{eqnarray*}
then
\begin{eqnarray*}
q^{-\frac{7}{4}}\mathcal{V}(-1,q^{-2})=-\dfrac{1}{2}\sum_{n\geq 0}n\;\psi(n)\;q^{\frac{n^2}{4}}
\end{eqnarray*}
where $\psi(n)$ is the (non-primitive) Dirichlet character modulo $2$ which is $0$ or $1$ according as $n$ is even or odd. Note that $\psi(n)$ is even with period $2$ and $S_\psi(k_0)=\mathcal{M}_{\psi}(k_0)=\{1\}$ where $k_0=1$. 
For $\alpha\in\mathbb{Q}$, it follows from Theorem \ref{main} that $\Theta_\psi(\alpha) := e^{-\frac{7\pi i\alpha}{2}}\;\mathcal{V}(-1,e^{-4\pi i\alpha})$ is a quantum modular form of weight $3/2$ on $A_2=\{\alpha\in\mathbb{Q}: \alpha\;\mbox{is}\;\Gamma_{1}(4)\mbox{-equivalent to}\;i\infty\}$ with respect to $\Gamma_1(4)$. Precisely, 
$\Theta_\psi(\alpha)$ satisfies 
\begin{eqnarray*}
\Theta_\psi(\alpha)-(\Theta_\psi|_{\frac{3}{2},\chi}\gamma)(\alpha)=r_{\gamma,\psi}(\alpha)
\end{eqnarray*}
for all $\gamma=\begin{pmatrix}
a & b\\c & d\end{pmatrix}\in\Gamma_1(4)$ and $\alpha\in A_2$, and where 
\begin{eqnarray*}
r_{\gamma,\psi}(z)=\dfrac{e^{\frac{\pi i}{4}}}{2\sqrt{2}\pi}\int_{\gamma^{-1}(i\infty)}^{i\infty}\theta_{\psi}(\tau)(\tau-\bar{z})^{-\frac{3}{2}}\;d\tau.
\end{eqnarray*}
Here, $r_{\gamma,\psi}:\mathbb{R}\rightarrow\mathbb{C}$ is a $C^{\infty}$ function which is real analytic in $\mathbb{R}\setminus\{\gamma^{-1}(i\infty)\}$, and $\chi$ is a multiplier given by
\begin{eqnarray*}
\chi(\gamma)=e^{\frac{\pi i ab}{2}} \left(\dfrac{4c}{d}\right)\varepsilon_d^{-1}.
\end{eqnarray*}

\subsection{Kontsevich-Zagier series for torus knots $T(2,2m+1)$}
Let $m\in\mathbb{N}$. For $0\leq \ell\leq m-1$, define the Kontsevich-Zagier series for the torus knot $T(2,2m+1)$ as follows:
\begin{eqnarray*}
X_m^{(\ell)}(q):=\sum_{k_1,k_2,\cdots,k_m=0}^\infty (q)_{k_m} q^{k_1^2+\cdots+k_{m-1}^2+k_{\ell+1}+\cdots+k_{m-1}}\prod_{i=1}^{m-1}\begin{bmatrix} k_{i+1} + \delta_{i,\ell} \\ k_{i} \end{bmatrix}
\end{eqnarray*}
where $\delta_{i,\ell}$ is the characteristic function. Hikami \cite{hikami2} established the strange identity
\begin{eqnarray}\label{strangehikami}
X_m^{(\ell)}(q)``=" -\dfrac{1}{2}\sum_{n=0}^\infty n \;\chi_{8m+4}^{(\ell)}(n)\;q^{\frac{n^2-(2m-2\ell-1)^2}{8(2m+1)}}
\end{eqnarray}
where 
\begin{equation*} \label{genchi}
\chi_{8m+4}^{(\ell)}(n):=
\begin{cases}
1 &\text{if $n \equiv 2m-2\ell-1$, $6m+2\ell+5$ $\pmod{8m+4}$,} \\
-1 &\text{if $n \equiv 2m+2\ell+3$, $6m-2\ell+1$ $\pmod{8m+4}$,} \\
0 &\text{otherwise.}
\end{cases}
\end{equation*}
Note that $f(n) :=\chi_{8m+4}^{(\ell)}(n)$ is an even function with period $8m+4$. For $k_0=(2m-2\ell-1)^2 \pmod{16m+8}$, consider the set $\mathcal{M}_{f}(k_0)=\{2m-2\ell-1, 2m+2\ell+3\}$. Thus, $S_{f}(k_0)=\{\pm (2m-2\ell-1), \pm (2m+2\ell+3)\}$.
Observe that $\sum_{k\in\mathcal{M}_{f}(k_0)} f(k)=0$. Thus, for $\alpha\in\mathbb{Q}$, Theorem \ref{main} and \eqref{strangehikami} imply that $e^{\frac{\pi i\alpha(2m-2\ell-1)^2}{8m+4}}X_m^{(\ell)}(e^{2\pi i\alpha})=-\dfrac{1}{2}\Theta_{f}(\alpha)=:\tilde{\Theta}_{m,\ell}(\alpha)$ is a quantum modular form of weight $3/2$ on $A_{8m+4}=\{\alpha\in\mathbb{Q}: \alpha\;\mbox{is}\;\Gamma_1(16m+8)\mbox{-equivalent to}\;0\;\mbox{or}\;i\infty\}$ with respect to $\Gamma_1(16m+8)$. Precisely, we have 
\begin{eqnarray*}
\tilde{\Theta}_{m,\ell}(\alpha)-\left(\tilde{\Theta}_{m,\ell}|_{\frac{3}{2},\chi}\gamma\right)(\alpha)=r_{\gamma,f}(\alpha)
\end{eqnarray*}
for all $\gamma=\begin{pmatrix}
a & b\\c & d\end{pmatrix}\in\Gamma_1(16m+8)$ and $\alpha\in A_{8m+4}$, where 
\begin{eqnarray*}
r_{\gamma, f}(z)=\dfrac{\sqrt{8m+4}\cdot e^{\frac{\pi i}{4}}}{4\pi}\int_{\gamma^{-1}(i\infty)}^{i\infty}\theta_{f}(\tau)(\tau-\bar{z})^{-\frac{3}{2}}\;d\tau.
\end{eqnarray*}
Here, $r_{\gamma, f}:\mathbb{R}\rightarrow\mathbb{C}$ is a $C^{\infty}$ function which is real analytic in $\mathbb{R}\setminus\{\gamma^{-1}(i\infty)\}$, and $\chi$ is a multiplier given by
\begin{eqnarray*}
\chi(\gamma)=e^{\frac{\pi i ab (2m-2\ell-1)^2}{(8m+4)}} \left(\dfrac{2c(8m+4)}{d}\right)\varepsilon_d^{-1}.
\end{eqnarray*}
We remark that Hikami proved $\tilde{\Theta}_{m,\ell}(z)$ is a vector-valued quantum modular form of weight $3/2$ on $SL_2(\mathbb{Z})$ (see \cite[page 195]{hikami2} for details). 

\subsection{Rogers' false theta function} For $M\in\mathbb{N}$ and $1\leq j<M$ with $j\neq \frac{M}{2}$, consider the false theta function of Rogers:
\begin{eqnarray*}
F_{j,M}(z)&:=&\sum_{n\equiv j\;(\text{mod}\;M)}\sgn(n)\;q^{\frac{n^2}{2M}}=\left(\sum_{\substack{n>0\\n\equiv j\;(\text{mod}\;M)}}-\sum_{\substack{n>0\\n\equiv -j\;(\text{mod}\;M)}}\right)q^{\frac{n^2}{2M}} = \sum_{n>0}f(n)\;q^{\frac{n^2}{2M}}
\end{eqnarray*}
where $f(n)$ is the function defined by $1$ or $-1$ according as $n\equiv j$ or $-j\;(\mbox{mod}\;M)$ and $0$ otherwise. Note that $F_{\frac{M}{2},M}(z)=0$. Here, $f$ is an odd function with period $M$. In this case, $\mathcal{M}(k_0)=\{j\}$ (respectively, $\mathcal{M}(k_0)=\{M-j\}$) for $1\leq j< \frac{M}{2}$ (respectively, $\frac{M}{2} <j<M$) with $k_0=j^2\;(\mbox{mod}\;2M)$ (respectively, $k_0=(M-j)^2\;(\mbox{mod}\;2M)$). So, $S_f(k_0)=\{j, M-j\}$. Thus, for $\alpha\in\mathbb{Q}$, Theorem \ref{main} implies that $F_{j,M}(\alpha)$ is a strong quantum modular form of weight $1/2$ on $\mathbb{Q}$ with respect to $\Gamma_{M}$ (given by (\ref{modd})). This result (with $z$ replaced by $\frac{z}{M}$ and $M$ even) was discussed in \cite[Theorem 4.1]{BM} (see \cite{BN} for a vector-valued version). More generally, for $1\leq k_0<2M$, if 
\begin{eqnarray*}
F_M(z):=\sum_{n>0}h(n)\;q^{\frac{n^2}{2M}}=
\sum_{j\in \mathcal{M}_h(k_0)}h(j)\;F_{j,M}(z)
\end{eqnarray*}
where $h(n)$ is an odd function with period $M$ and support $S_h(k_0)$, then Theorem \ref{main} shows that $F_M(z)$ is a strong quantum modular form of weight $1/2$ on $\mathbb{Q}$ with respect to $\Gamma_{M}$. Finally, $F_{j,M}(\alpha)$ and, more generally, $F_M(\alpha)$ are quantum modular forms of weight $1/2$ on $B_M$ with respect to $\Gamma_M$.  

\section*{Acknowledgements}
The first author is supported by grant SFB F50-06 of the Austrian Science Fund (FWF). The second author would like to thank the Max-Planck-Institut f\"ur Mathematik for their support during the initial stages of this project, the Ireland Canada University Foundation for the James M. Flaherty Visiting Professorship award and McMaster University for their hospitality during his stay from May 17 to August 9, 2019. The second author also thanks Yingkun Li for a clarifying remark during a visit to TU Darmstadt on April 23, 2019. Finally, the authors thank Jeremy Lovejoy for insightful comments on a preliminary version of this paper, Sergei Gukov for kindly reminding us of \cite{g} and the referee for helpful comments and suggestions.


\begin{thebibliography}{999}

\bibitem{bd} 
S. Bettin, S. Drappeau, \emph{Modularity and value distribution of quantum invariants of hyperbolic knots}, preprint available at \url{https://arxiv.org/abs/1905.02045}

\bibitem{Be}
C. Bijaoui, H.U. Boden, B. Myers, R. Osburn, W. Rushworth, A. Tronsgard and S. Zhou, \emph{Generalized Fishburn numbers and torus knots}, J. Combin. Theory Ser. A \textbf{178} (2021), 105355.

\bibitem{bcr}
K. Bringmann, T. Creutzig and L. Rolen, \emph{Negative index Jacobi forms and quantum modular forms}, Res. Math. Sci. \textbf{1} (2014), Art. 11, 32pp.

\bibitem{bfor}
K. Bringmann, A. Folsom, K. Ono and L. Rolen, \emph{Harmonic Maass forms and mock modular forms: theory and applications}, American Mathematical Society Colloquium Publications, 64. American Mathematical Society, Providence, RI, 2017.

\bibitem{bmm}
K. Bringmann, K. Mahlburg and A. Milas, \emph{Quantum modular forms and plumbing graphs of $3$-manifolds}, J. Combin. Theory Ser. A \textbf{170} (2020), 105145, 32pp.

\bibitem{BM}
K. Bringmann, A. Milas, \emph{$\mathscr{W}$-algebras, false theta functions and quantum modular forms, I}, Int. Math. Res. Not. IMRN 2015, no. 21, 11351--11387.

\bibitem{BN}
K. Bringmann, C. Nazaroglu, \emph{A framework for modular properties of false theta functions}, Res. Math. Sci. \textbf{6} (2019), no. 3, Paper No. 30, 23pp.

\bibitem{BR}
K. Bringmann, L. Rolen, \emph{Half-integral Eichler integrals and quantum modular forms}, J. Number Theory \textbf{161}(2016), 240--254.

\bibitem{b}
R. Bruggeman, \emph{Quantum Maass forms}, The Conference on $L$-functions, 1--15, World Sci. Publ., Hackensack, NJ, 2007.

\bibitem{ccfgh}
M. Cheng, S. Chun, F. Ferrari, S. Gukov and S. Harrison, \emph{$3d$ Modularity}, J. High Energ. Phys. \textbf{2019}, 10 (2019).

\bibitem{cfs}
M. Cheng, F. Ferrari and G. Sgroi, \emph{Three-manifold quantum invariants and mock theta functions}, Philos. Trans. Roy. Soc. A \textbf{378} (2020), no. 2163, 20180439, 15pp.

\bibitem{djr}
A. Dabholkar, D. Jain and A. Rudra, \emph{APS $\eta$-invariant, path integrals, and mock modularity}, J. High Energ. Phys. \textbf{2019}, 80 (2019).

\bibitem{fjks}
A. Folsom, M.-J. Jang, S. Kimport and H. Swisher, \emph{Quantum modular forms and singular combinatorial series with repeated roots of unity}, Acta Arith. \textbf{194} (2020), no. 4, 393--421.

\bibitem{g}
S. Gukov, \emph{Three-dimensional quantum gravity, Chern-Simons theory, and the $A$-polynomial}, Comm. Math. Phys. \textbf{255} (2005), no. 3, 577--627.

\bibitem{habiro}
K. Habiro, \emph{On the colored Jones polynomial of some simple links}, in: Recent progress toward the volume conjecture (Kyoto, 2000), S{\=u}rikaisekikenky{\=u}sho K{\=o}ky{\=u}roku \textbf{1172} (2000), 34--43.

\bibitem{hikami2}
K. Hikami, \emph{$q$-series and $L$-functions related to half-derivates of the Andrews-Gordon identity}, Ramanujan J. \textbf{11} (2006), no. 2, 175--197.

\bibitem{hikami3}
K. Hikami, \emph{Quantum invariants, modular forms, and lattice points. II}, J. Math. Phys. \textbf{47} (2006), no. 10, 102301, 32pp.

\bibitem{hk2}
K. Hikami, A. Kirillov, \emph{Hypergeometric generating function of $L$-function, Slater's identities, and quantum invariant}, St. Petersburg Math. J. \textbf{17} (2006), no. 1, 143--156.

\bibitem{hl2}
K. Hikami, J. Lovejoy, \emph{Hecke-type formulas for families of unified Witten-Reshetikhin-Turaev invariants}, Commun. Number Theory Phys. \textbf{11} (2017), no. 2, 249--272.

\bibitem{kll}
B. Kim, S. Lim and J. Lovejoy, \emph{Odd-balanced unimodal sequences and related functions: parity, mock modularity and quantum modularity}, Proc. Amer. Math. Soc. \textbf{144} (2016), no. 9, 3687--3700.

\bibitem{konan}
I. Konan, \emph{Autour des $q$-s{\'e}ries, des formes modulaires quantiques et des n{\oe}uds toriques}, Master's thesis, Universit{\'e} Denis Diderot - Paris 7, available at \url{https://www.irif.fr/_media/users/konan/memoire.pdf}

\bibitem{LZ}
R. Lawrence, D. Zagier, \emph{Modular forms and quantum invariants of 3-manifolds}, Asian J. Math. \textbf{3} (1999), no. 1, 93--107.

\bibitem{thang}
T. T. Q. L{\^e}, \emph{Quantum invariants of 3-manifolds: Integrality, splitting, and perturbative expansion}, Topology Appl. \textbf{127} (2003), no. 1-2, 125--152.

\bibitem{lz}
J. Lewis, D. Zagier, \emph{Cotangent sums, quantum modular forms, and the generalized Riemann hypothesis}, Res. Math. Sci. \textbf{6} (2019), no. 1, Paper No. 4, 24pp.

\bibitem{my}
H. Murakami, Y. Yokota, \emph{Volume conjecture for knots}, SpringerBriefs in Mathematical Physics, 30. Springer, Singapore, 2018.

\bibitem{n}
A. Nordentoft, \emph{A note on additive twists, reciprocity laws and quantum modular forms}, Ramanujan J., to appear.

\bibitem{Rad}
H. Rademacher, \emph{Topics in analytic number theory}, edited by E. Grosswald, J. Lehner and M. Newman. Die Grundlehren der mathematischen Wissenschaften, Band 169. Springer-Verlag, New York-Heidelberg, 1973.

\bibitem{shimura}
G. Shimura, \emph{On modular forms of half integral weight}, Ann. of Math. (2) \textbf{97} (1973), 440--481.

\bibitem{z1}
D. Zagier, \emph{Vassiliev invariants and a strange identity related to the Dedekind eta-function}, Topology {\bf 40} (2001), no. 5, 945--960.

\bibitem{z}
D. Zagier, \emph{Quantum modular forms}, Quanta of maths, 659--675, Clay Math. Proc., \textbf{11}, Amer. Math. Soc., Providence, RI, 2010.

\end{thebibliography}
\end{document}